\documentclass{amsart}
\usepackage{latexsym, amsmath, amsfonts, amssymb, geometry, mathrsfs, fancyhdr, tikz, tikz-cd, verbatim, dsfont}
\usetikzlibrary{matrix,arrows,decorations.pathmorphing,calc}

\def\a{\alpha}

\def\la{\langle}
\def\ra{\rangle}

\renewcommand{\phi}{\varphi}
\newcommand{\Z}{\mathbb{Z}}

\newtheorem{theorem}{Theorem}

\newtheorem{cor}{Corollary}[theorem]

\theoremstyle{definition}
\newtheorem{definition}{Definition}
\newtheorem{example}[theorem]{Example}

\theoremstyle{remark}
\newtheorem*{remark}{Remark}

\numberwithin{equation} {section}

\begin{document}

\title{On the Symmetric Difference Property in Difference Sets under Product Construction}

\author{Andrew Clickard}
\address{Department of Mathematical and Digital Sciences, Bloomsburg University, Bloomsburg, Pennsylvania 17815}

\email{ac24869@huskies.bloomu.edu}
\thanks{Thank you to Drs. James Davis, John Polhill, and Ken Smith for their counsel and teaching.}

\date{\today.}


\begin{abstract}
A $(v, k, \lambda)$ symmetric design is said to have the symmetric difference property (SDP) if the symmetric difference of any three blocks is either a block or the complement of a block. Symmetric designs fulfilling this property have the nice property of having minimal rank, which makes them interesting to study. Thus, SDP designs become useful in coding theory applications. We show in this paper that difference sets formed by direct product construction of difference sets whose developments have the SDP also have the SDP. We also establish a few results regarding isomorphisms in product constructed SDP designs.
\end{abstract}

\maketitle




\section{Introduction}\label{Introduction}
Extremal error correcting codes, (codes whose parameters meet a bound)  have long been studied in the coding theory community, as codes that optimize the minimal distance between codes allows for more errors to be corrected. One way to generate these codes is to consider symmetric designs and their incidence matrices. From this effort, the concept of the symmetric difference property (SDP) was established; a property that minimizes the rank of the incidence matrix. The properties and interactions that SDP designs hold are often obfuscated by having to interact with extremely large matrices, and so there is a great deal that is simply not known. This paper seeks to extend the knowledge of designs with this property, particularly SDP designs coming from difference sets in groups of order $2^{2n}$.
The main results of this paper are as follows:
\begin{enumerate}
    \item If $D_1,D_2$ are difference sets in groups $G_1,G_2$, then $D = (D_1 \times (G_2-D_2)) \cup ((G_1-D_1) \times D_2)$ has the SDP if and only if $D_1$ and $D_2$ have the SDP. (Section \ref{section:direct-product}, Theorem \ref{theorem:direct-product})
    \item Given symplectic difference sets $D_1$ and $D_2$ in groups $G_1$ and $G_2$ and a homomorphism $\phi:G_2 \to Aut(G_2)$, the product construction $D$ of $D_1$ and $D_2$ is a symplectic difference set in $G_1 \rtimes_\phi G_2$ if the automorphism induced by each generator of $G_2$ under $\phi$ fixes $D_1$. (Section \ref{section:semi-direct}, Theorem \ref{theorem:semi-direct condition})
    \item If $D,D'$ are the product construction of SDP difference sets $D_1,D_2$ and $D_1',D_2'$ in $G_1 \times G_2$ and $G_1' \times G_2'$ with $|G_1| = |G_1'|$, $|G_2| = |G_2'|$, then the developments of $D$ and $D'$ are isomorphic if the developments of $D_1$ and $D_1'$ are isomorphic and the developments of $D_2$ and $D_2'$ are isomorphic. (Section \ref{section:isomorphisms}, Theorem \ref{theorem:direct-iso})
\end{enumerate}
The symmetric difference property being closed under direct product construction allows for a whole new line of questioning: In what groups is it possible to produce SDP difference sets from product construction? And more generally, we can ask, given two difference sets with the SDP, what effects do different semi-direct products have on the product construction of the sets?


\section{Preliminaries}\label{section:Notation}
Let us first lay out some notation and definitions to be used going forward.

\begin{definition}
Let $A,B$ be sets. Then the operation $\Delta$ is defined by $A \Delta B = (A-B)\cup(B-A)$ and is called the \textit{symmetric difference} of $A$ and $B$.
\end{definition}
The sets we want to take the symmetric differences of are the blocks of a symmetric design, which is defined thusly:
\begin{definition}
A $(v,k,\lambda)$ symmetric design is a set of points $P$ together with a set of blocks $B$ such that there are $v$ points and $v$ blocks such that any point is incident on $k$ blocks, any block is incident on $k$ points, any two points share incidence on $\lambda$ blocks, and any two blocks share incidence on $\lambda$ points.
\end{definition}
Importantly for this paper, we may organize a symmetric design into what is know as an incidence matrix, which is defined here:
\begin{definition}
The incidence matrix of a $(v,k,\lambda)$ symmetric design with point set $P$ and block set $B$ is the $v \times v$ matrix with columns labelled by the elements of $P$ and rows labelled by the elements of $B$, with the $ij$-th entry being 1 if the $j$-th point is incident on the $i$-th block, and 0 otherwise.
\end{definition}
Now, we are most interested in symmetric designs that come from groups, and in particular, the subsets of groups known as difference sets:
\begin{definition}
A $(v,k,\lambda)$ difference set is a set $D$ of order $k$ in a group $G$ of order $v$ such that the multiset $\{d_1d_2^{-1} \ | \ d_1,d_2 \in D\}$ contains every non-identity element of the group $\lambda$ times. The \textit{development} of a difference set $D$ is the symmetric design with the elements of $G$ as points, the left translates of $D$ by the elements of $G$ as blocks, and the incidence relation defined by set inclusion.
\end{definition}
And now, we may finally define the symmetric difference property in symmetric difference sets.

\begin{definition}
A $(v,k,\lambda)$ symmetric design has the \textit{symmetric difference property} (SDP) if the symmetric difference of any three blocks is either a block or the complement of a block. For further motivation and reading regarding this property, see \cite{Kantor}.
\end{definition}

Note that if we consider the incidence matrix of a symmetric design, having the SDP is equivalent to saying that the addition of three rows (addition here defined as element-wise addition modulo 2) of the matrix is either a row or the complement of a row. This is the definition to be used throughout the remainder of the paper. The complement of the incidence matrix $A$ is defined to be $A^c = A + J$, where $J$ is the all-ones matrix of the same size as $A$. 

\begin{definition}
If $D_1=(P_1,B_1), \ D_2=(P_2,B_2)$ are two $(v,k,\lambda)$ symmetric designs, they are isomorphic if and only if there exists a bijection $\phi:P_1 \to P_2$ that preserves incidence. Equivalently, let $A_1,A_2$ be the incidence matrices for $D_1,D_2$, respectively. Then $D_1$ and $D_2$ are isomorphic if and only if there exist permutation matrices $P,Q$ such that $A_1 = PA_2Q$.
\end{definition}

One of the earliest examples of an error-correcting code is the Reed-Muller code, which was first introduced in 1954 by the titular mathematicians Reed and Muller in \cite{Reed} and \cite{Muller}. Our naming scheme for the basis elements of the first-order Reed-Muller code $RM(1,m)$ is as follows. The first $m$ basis elements have $2^{m-1}$ zeros and $2^{m-1}$ ones. The first basis element $c_1$ has the full string of zeros followed by the full string of ones. $c_2$ has a string of $2^{m-2}$ zeros, $2^{m-2}$ ones, and repeats that pattern again. In general, the $i$-th basis element has $2^{i}$ alternating strings of $2^{m-i}$ zeros and ones. The $(m+1)$-th basis word, $\mathds{1}$ is the all-ones word, and is denoted as such. As an illustrative example, the basis words for $RM(1,4)$ are:
\[
\begin{matrix}
c_1: & 0000 & 0000 & 1111 & 1111\\
c_2: & 0000 & 1111 & 0000 & 1111\\
c_3: & 0011 & 0011 & 0011 & 0011\\
c_4: & 0101 & 0101 & 0101 & 0101\\
\mathds{1}: & 1111 & 1111 & 1111 & 1111
\end{matrix}
\]
The Reed-Muller codes may also be defined recursively by $RM(1,m)=\{\la u \ \ |\ v \ra \ | \ u,v \in RM(1,m)\}$, where $\la u \ |\ v \ra$ denotes the concatenation of strings, but the basis element definition becomes a useful construction in the proof of Theorem \ref{theorem:direct-product}.

Section \ref{section:semi-direct} deals with semi-direct products of difference sets with developments that are isomorphic to the symplectic design on $2^{2n}$ points. The symplectic design on $2^{2n}$ points, first developed by Kantor in \cite{Kantor}, is the design formed by iterative product construction of the trivial SDP in $C_2^2$, and thus has incidence matrix $A$ defined by
\[A = \frac{-1}{2}\bigg(\big(\underbrace{(J_4-2I_4) \otimes (J_4-2I_4) \otimes \cdots \otimes (J_4-2I_4)}_{n \text{ times}}\big)-J_4\bigg),\]
where $\otimes$ denotes the Kronecker product, $I_4$ is the $4 \times 4$ identity matrix and $J_4$ is the $4 \times 4$ all-ones matrix. The definition for the semi-direct product of groups is as follows:
\begin{definition}
Let $N$, $H$ be groups, $Aut(N)$ be the automorphism group of $N$, and $\phi : H \to Aut(N)$ a homomorphism. Then the semi-direct product $N \rtimes_\phi H$ of $N$ and $H$ by $\phi$ is the group which has underlying set $N \times H$ and operation defined by $(n_1,h_1)(n_2,h_2) = (n_1\phi(h_1)(n_2), h_1h_2)$.
\end{definition}

The cyclic group of order $n$ is denoted by $C_n$, and is written multiplicatively. The notation $C_n^m$ denotes the direct product of $m$ copies of $C_n$.

The last piece of background we need to address upon which this paper is based is the product construction of difference sets. First, we must discuss doing calculations in the group ring $\Z[G]$. There are two main ways of representing difference sets in this group ring, each of which has its own set of strengths and weaknesses. The one we use for illustration purposes with a difference set $D$ in a group $G$ is defined by $D = \sum_{g \in G}(-1)^jg$, where $j = 0$ if $g \not\in D$ and $j=1$ otherwise. We also define $D^{(-1)} = \sum_{g \in G}(-1)^jg^{-1}$, where $j = 0$ if $g \not\in D$ and $j=1$ otherwise (hence, the element-wise inverses). From \cite{Dillon}, we have that if $D$ is a difference set in $G$, where $G$ is a 2-group, then $DD^{(-1)} = |G|$. Now, if $D_1,D_2$ are difference sets in the groups $G_1,G_2$, then the set $D = \big( D_1 \times (G_2-D_2) \big) \cup \big((G_1-D_1) \times D_2 \big)$ is a difference set in the group $G_1 \times G_2$. In the group ring $\Z[G_1 \times G_2]$, we let $D = D_1D_2$, where $D_1$ and $D_2$ are the sums as above over $G_1$ and $G_2$, respectively. We then see that $DD^{(-1)} = D_1D_2D_1^{(-1)}D_2^{(-1)} = |G_1||G_2| = |G_1 \times G_2|$, and thus $D$ is a difference set in $G_1 \times G_2$ which, written in standard set theory notation, is $\big( D_1 \times (G_2-D_2) \big) \cup \big((G_1-D_1) \times D_2 \big)$. Since the underlying set is $G_1 \times G_2$ for any semi-direct product, this construction works for any homomorphism $\phi:G_2 \to Aut(G_1)$.


\section{Closure of the SDP under Direct Product Construction}\label{section:direct-product}

We first begin with the closure of the SDP under direct product, followed by a discussion of the effects of group ``factoring" on the resulting product construction. A very specific case of the following theorem is found in \cite{Davis}, which showed that the SDP is closed when product constructed with the trivial difference set in $C_4$.

\begin{theorem}\label{theorem:direct-product}
Let $G = G_1 \times G_2$ be a group of order $2^{2n}$ with $|G_1|=2^{v_1}$, $|G_2|=2^{v_2}$, with $v_1$, $v_2$ even. Let $D_1$ and $D_2$ be difference sets in $G_1,G_2$, respectively. Then $D = \big(D_1 \times (G_2 - D_2)\big) \cup \big((G_1-D_1) \times D_2\big)$ has the SDP if and only if $D_1$ and $D_2$ have the SDP.
\end{theorem}
\begin{proof}
That $D$ is a difference set follows immediately from the product construction of difference sets, so all we must just verify is that $D$ has the SDP. Let $A_1$ and $A_2$ be the incidence matrices of the developments of $D_1$ and $D_2$, and keep the ordering of $G_1$ as used in $A_1$ and likewise for $G_2$ consistent throughout this proof. Let $G$ have the ordering given by $G_1\times\{g_1\}, G_1\times\{g_2\}, \cdots, G_1\times\{g_{|G_2|}\}$, where $g_i$ is the $i$-th element of $G_2$ in its ordering and $G_1$ is internally ordered as in $A_1$, and let $A$ be the development of $D$ with this ordering of $G$.

From here, consider the sub-matrices formed by the rows labelled by $G_1 \times \{g_i\}$ and the columns labelled by $G_1 \times \{g_j\}$. Firstly, note that we are now considering $A$ as a $2^{v_2} \times 2^{v_2}$ block matrix with $2^{v_1}\times2^{v_1}$ matrices as entries. Note that if $g_j \in g_iD_2$, then the sub-matrix will be equal to $A_1^c$, since we are forced to consider incidence in $(G_1-D_1) \times D_2$. But if $g_j \not\in g_iD_2$, then the sub-matrix is identical to $A_1$ by the same reasoning. Let $A'$ denote the $2^{v_2} \times 2^{v_2}$ block matrix with $A_1$ in every entry, and $A_2'$ denote the $2^{v_2} \times 2^{v_2}$ block matrix whose $ij$-th matrix is the all-ones matrix $J$ if the $ij$-th entry of $A_2$ is 1, and the zero matrix otherwise. Then we see that $A = A' + A_2'$, and thus $A$ has a similar structure to $A_2$.

Since $D_2$ has the SDP, this implies that the addition of any three of these block rows is either a block row or the complement of a block row. Let $R_i,R_j,R_k$ be three arbitrary block rows of $A$. Then $R_i + R_j + R_k \in \{R_u,R_u^c\}$ for some $u$. Let $r_i,r_j,r_k$ be rows contained in $R_i, R_j, R_k$, respectively. If $R_i=R_j=R_k$, then since $D_1$ has the SDP, and $r_i, r_j, r_k$ are the concatenation of $2^{v_2}$ copies of rows in $A_1$, then $r_i + r_j + r_k \in \{r_v, r_v^c\}$ for some $r_v \in R_i$. If, without loss of generality, $R_i=R_j$, then $R_i + R_j + R_k = R_k$. Since, as was discussed in Section \ref{section:Notation}, each row of $A_1$ defines a bent function in $RM(1,v_1)$, the addition $r_i + r_j$ is in $RM(1,2n)$. Therefore $r_i + r_j + r_k \in \{r_v,r_v^c\}$ for some $r_v \in R_k$. The other equalities follow by symmetry. If all three block rows are pairwise unequal, we consider the following. Let $r_i,r_j,r_k$ be arbitrary rows of the block rows $R_i,R_j,R_k$, respectively. Note that $r_i,r_j,r_k$ are each concatenations of a combination of rows $\rho_i,\rho_j,\rho_k$ of $A_1$ and their complements based on the $i$-th, $j$-th, and $k$-th row of $A_2$ (denote these rows $\a_i,\a_j,\a_k$, and the large versions of these rows $a_i,a_j,a_k$), respectively. Since $D_1$ has the SDP, this implies that the addition of each of these concatenated rows will result in the concatenation of some combination of a row $\rho_v$ and its complement. Then 
\[r_i + r_j + r_k = 
\begin{matrix}
&\la& \rho_i & | & \rho_i & | & \cdots & | & \rho_i &\ra + a_i \\
&\la& \rho_j & | & \rho_j & | & \cdots & | & \rho_j &\ra + a_j  \\
+&\la& \rho_k & | & \rho_k & | & \cdots & | & \rho_k &\ra + a_k  \\
\end{matrix}\]
Which is one of \[\left\{\la \rho_v \ | \ \rho_v \ | \ \cdots \ | \ \rho_v \ra,\la  \rho_v^c \ |\  \rho_v^c \ | \ \cdots \ | \ \rho_v^c \ra  \right\} + \left\{ a_v,a_v^c\right\},\]
all of which are a row of $A$ or the complement thereof. Thus, the addition of any three rows of $A$ is either a row of $A$ or the complement of a row, and so $D$ has the SDP.

Conversely, without loss of generality suppose that $D_1$ does not have the SDP. By way of contradiction, suppose $D$ has the SDP. Consider the sub-matrix $M$ formed by the rows of the form $r_1 + \{RM(1,2n)-\langle c_1,\cdots,c_{v_2}\rangle\}$ and the columns labelled by $G_1 \times \{1\}$. Firstly, note that $M$ is a $2^{v_1} \times 2^{v_1}$ square matrix. By construction, the first row $\rho_1$ of $M$ must be the incidence of either $D_1$ or $G_1-D_1$. Also, since $D$ has the SDP, and $M$ is formed by a subgroup of the code, the addition of any three rows of $M$ must also be a row or the complement of a row, which implies that the incidence of $D_1$ or $G_1-D_1$ is a bent function on $RM(1,v_1)$. Thus, the development of $D_1$ must have the SDP contrary to supposition, and therefore $D$ cannot have the SDP.
\end{proof}

Note that this implies that the possible product-constructed SDP difference sets in $G$ depends on the ``factoring" of $G$. For example, in the Abelian group $C_8 \times C_8 \times C_2 \times C_2$, there are no SDP difference sets coming from the grouping $(C_8^2) \times (C_2^2)$, but there is coming from $(C_8 \times C_2)^2$. From this, we can state the following corollaries:

\begin{remark}
The Abelian group $C_8 \times C_8 \times C_4$ has no SDP difference sets from product construction.
\end{remark}
\begin{proof}
The only eligible grouping of $C_8 \times C_8 \times C_4$ into difference set-containing groups is $(C_8)^2 \times C_4$, so since $C_8^2$ has no SDP difference sets, $C_8 \times C_8 \times C_4$ cannot have any SDP difference sets by product construction by Theorem \ref{theorem:direct-product}.
\end{proof}

\begin{remark}\label{corollary:c16}
The abelian groups $C_{16} \times (C_8 \times C_2)$, $C_{16} \times C_4^2$,  $C_{16} \times (C_8 \times C_2) \times (C_4) \times (C_2^2)$, when grouped as such, have no SDP difference sets from product construction for $i,j,k \geq 0$.
\end{remark}
\begin{proof}
None of $C_{16} \times (C_8 \times C_2)$, $C_{16} \times C_4$, nor $C_{16} \times (C_2^2)$ contain SDP difference sets from product construction by Theorem \ref{theorem:direct-product}, from which it immediately follows that any grouping of these groups has no SDP difference set from product construction.
\end{proof}

\begin{remark}
This grouping as in the last remark is some sense the ``best" factoring of the group. That is, it is grouped so as to maximize the number of SDP-containing groups and placing priority on larger ``primitive" (not product constructed) SDP difference set-containing groups. Another example of this factoring is the factoring of the group $C_8 \times C_8 \times C_2 \times C_2$ as we discussed earlier.
\end{remark}

With this result under our belts, we may now expand to the more general semi-direct product.

\section{Closure of the SDP under Semi-Direct Product Construction}\label{section:semi-direct}

In general, the SDP will not be closed under product construction. Section \ref{section:direct-product} is dedicated to showing a very specific case of the general product construction. The question then becomes: under what semi-direct products does the product construction of two SDP difference sets yield the SDP? Note that the SDP is not always preserved under the general product, as, for example, in the groups of the form $(C_8 \times C_2) \rtimes (C_8 \times C_2)$, there are seven distinct (nonisomorphic) designs. Of these, only two have the SDP. For the purpose of illustration, the following table shows the mappings of $x$ and $y$ that produces each design (and so the homomorphism will be defined by a choice of mapping for $x$ and a choice of mapping for $y$). Note that for designs 3 and 4, the combinations of the mappings of $x$ and $y$ are restricted. These restrictions for Design 3 will be discussed following the table.
\begin{center}
\begin{tabular}{c|c c || c|c c}
    \hline\hline
    Design 1 & Rank: 10 & Total Designs: 16 & Design 2 & Rank: 10 & Total Designs: 8\\\hline
    $x\mapsto$ &  $\phi_x(z) = z,$  & $\phi_x(w) = w$ & $x\mapsto$ &  $\phi_x(z) = zw,$  & $\phi_x(w) = w$ \\
    &  $\phi_x(z) = z^3,$  & $\phi_x(w) = z^4w$ & &  $\phi_x(z) = z^3w,$  & $\phi_x(w) = z^4w$\\
    &  $\phi_x(z) = z^5,$  & $\phi_x(w) = w$ & &  $\phi_x(z) = z^5w,$  & $\phi_x(w) = w$\\
    &  $\phi_x(z) = z^7,$  & $\phi_x(w) = z^4w$ & &  $\phi_x(z) = z^7w,$  & $\phi_x(w) = z^4w$\\
    \hline
    $y\mapsto$ & $\phi_y(z) = z,$  & $\phi_y(w) = w$ & $y\mapsto$ & $\phi_y(z) = z,$  & $\phi_y(w) = w$ \\
    &  $\phi_y(z) = z^3,$  & $\phi_y(w) = z^4w$&&\\
    &  $\phi_y(z) = z^5,$  & $\phi_y(w) = w$ && $\phi_y(z) = z^5,$  & $\phi_y(w) = w$\\
    &  $\phi_y(z) = z^7,$  & $\phi_y(w) = z^4w$&&\\\hline\hline

    Design 3 & Rank: 11 & Total Designs: 24 & Design 4 & Rank: 12 & Total Designs: 40\\\hline
    $x\mapsto$ & $\phi_x(z)= z$ & $\phi_x(w)= z^4w$ & $x\mapsto$ & $\phi_x(z)= z$ & $\phi_x(w)= w$ \\
    & $\phi_x(z)= z^3$ & $\phi_x(w)= w$ & & $\phi_x(z)= z$ & $\phi_x(w)= z^4w$ \\
    & $\phi_x(z)= z^5$ & $\phi_x(w)= z^4w$ & & $\phi_x(z)= z^3$ & $\phi_x(w)= w$ \\
    & $\phi_x(z)= z^7$ & $\phi_x(w)= w$ & & $\phi_x(z)= z^3$ & $\phi_x(w)= z^4w$ \\
    & $\phi_x(z)= zw$ & $\phi_x(w)= z^4w$ & & $\phi_x(z)= z^5$ & $\phi_x(w)= w$\\
    & $\phi_x(z)= z^3w$ & $\phi_x(w)= w$ & & $\phi_x(z)= z^5$ & $\phi_x(w)= z^4w$\\
    & $\phi_x(z)= z^5w$ & $\phi_x(w)= z^4w$ & & $\phi_x(z)= z^7$ & $\phi_x(w)= w$\\
    & $\phi_x(z)= z^7w$ & $\phi_x(w)= z$ & & $\phi_x(z)= z^7$ & $\phi_x(w)= z^4w$\\
    &&&& $\phi_x(z)= z^3w$ & $\phi_x(w)= w$\\
    &&&& $\phi_x(z)= z^7w$ & $\phi_x(w)= w$\\
    \hline

    $y\mapsto$ & $\phi_y(z)= z$ & $\phi_y(w)= w$ & $y\mapsto$ & $\phi_y(z)= z$ & $\phi_y(w)= z^4w$\\
    & $\phi_y(z)= z^3$ & $\phi_y(w)= z^4w$ & & $\phi_y(z)= z^3$ & $\phi_y(w)= w$\\
    & $\phi_y(z)= z^5$ & $\phi_y(w)= w$ & & $\phi_y(z)= z^5$ & $\phi_y(w)= z^4w$\\
    & $\phi_y(z)= z^7$ & $\phi_y(w)= z^4w$ & & $\phi_y(z)= z^7$ & $\phi_y(w)= w$\\
    &&&& $\phi_y(z)= z^3w$ & $\phi_y(w)= w$\\
    &&&& $\phi_y(z)= z^7w$ & $\phi_y(w)= w$\\
    \hline\hline
\end{tabular}

\begin{tabular}{c|c c || c|c c}
    Design 5 &Rank: 12 & Total Designs: 24\\\hline
    $x\mapsto$ & $\phi_x(z)= z^3$ & $\phi_x(w)= w$ & $y\mapsto$ & $\phi_y(z)= z^3$ & $\phi_y(w)= w$ \\
    & $\phi_x(z)= z^7$ & $\phi_x(w)= w$ & & $\phi_y(z)= z^7$ & $\phi_y(w)= w$ \\
    & $\phi_x(z)= zw$ & $\phi_x(w)= w$ & & $\phi_y(z)= z^3w$ & $\phi_y(w)= w$ \\
    & $\phi_x(z)= z^3w$ & $\phi_x(w)= z^4w$ & & $\phi_y(z)= z^7w$ & $\phi_y(w)= w$ \\
    & $\phi_x(z)= z^5w$ & $\phi_x(w)= z^4w$ & & \\
    & $\phi_x(z)= z^7w$ & $\phi_x(w)= w$ & & \\
    \hline\hline


    Design 6 &Rank: 11&Total Designs: 8& Design 7 & Rank: 11&Total Designs: 8\\\hline
    $x\mapsto$ &  $\phi_x(z) = z,$  & $\phi_x(w) = w$ & $x\mapsto$ &  $\phi_x(z) = z^3,$  & $\phi_x(w) = w$\\
    &  $\phi_x(z) = z^5,$  & $\phi_x(w) = w$ & &  $\phi_x(z) = z^7,$  & $\phi_x(w) = w$\\
    &  $\phi_x(z) = zw,$  & $\phi_x(w) = w$ & &  $\phi_x(z) = z^3w,$  & $\phi_x(w) = w$\\
    &  $\phi_x(z) = z^5w,$  & $\phi_x(w) = w$ & &  $\phi_x(z) = z^7w,$  & $\phi_x(w) = w$\\
    \hline
    $y\mapsto$ & $\phi_y(z) = zw,$  & $\phi_y(w) = w$ & $y\mapsto$ & $\phi_y(z) = zw,$  & $\phi_y(w) = w$ \\
    &  $\phi_y(z) = z^5w,$  & $\phi_y(w) = w$ & &  $\phi_y(z) = z^5w,$  & $\phi_y(w) = w$\\\hline
\end{tabular}
\end{center}
When discussing the maps of $x$ and $y$, for brevity's and notation's sake we refer only to $\phi_x(z)$ and $\phi_y(z)$. In Design 3, any of the maps of the form $\phi_x(z)=z^i$ may be freely combined with any map of $y$, but the maps of the form $\phi_x(z)=z^iw$ can only be combined with $\phi_y(z)=z$ and $\phi_y(z)=z^5$.

Let us narrow the purview of the discussion to semi-direct products of difference sets whose developments are isomorphic to the symplectic design (henceforth referred to as ``symplectic difference sets"). Toward this, we first note that for groups $G_1$ and $G_2$ with SDP difference sets $D_1$ and $D_2$, respectively, the product construction $D$ of $D_1$ and $D_2$ in $G_1 \rtimes G_2$ is the same (set-equivalent) regardless of choice of homomorphism from $G_2$ into $Aut(G_1)$, as the homomorphism only affects the multiplication of \textit{elements}, and no such operation is occurring in the product construction. Thus, since the difference sets are element-wise equal, the question must come down to the blocks of the design, since this is the only part of the design that relies upon multiplication within each group.

Recall that the incidence matrix for the symplectic design on $2^{2n}$ points is \[A = \frac{-1}{2}\bigg(\big(\underbrace{(J_4-2I_4) \otimes (J_4-2I_4) \otimes \cdots \otimes (J_4-2I_4)}_{n \text{ times}}\big)-J_4\bigg),\] where $\otimes$ denotes the Kronecker product, $I_4$ is the $4 \times 4$ identity matrix, and $J_4$ is the $4 \times 4$ all-ones matrix. Let $D_1$ and $D_2$ be symplectic difference sets in $G_1$ and $G_2$, respectively, and let $\phi : G_2 \to Aut(G_1)$ be a homomorphism. Note that if $ker(\phi)=G_2$, then $G_1 \rtimes_\phi G_2 = G_1 \times G_2$, and $D = (D_1 \times (G_2-D_2)) \cup ((G_1-D_1) \times D_2)$ has the SDP by Theorem \ref{theorem:direct-product}, and in fact is the symplectic design, since the construction in Theorem \ref{theorem:direct-product} is homologous to the Kronecker product construction of the incidence matrix.

\begin{theorem}\label{theorem:semi-direct condition}
Let $D_1$ and $D_2$ be symplectic SDP difference sets in $G_1$ and $G_2$, respectively. Let $\phi:G_2 \to Aut(G_1)$ be a homomorphism, and let $G = G_1 \rtimes_\phi G_2$ be the semi-direct product of $G_1$ by $G_2$ under $\phi$. Then $D = \big(D_1 \times (G_2-D_2)\big) \cup \big((G_1-D_1) \times D_2\big)$ is a symplectic difference set if $(g_i,g_j)*D = (g_i,g_j) \cdot D$ for all $(g_i,g_j) \in G_1 \times G_2$, where $*$ is the operation defined by $\phi$, and $\cdot$ is the operation of the direct product.
\end{theorem}
\begin{proof}
Let $\phi$ be such that $(g_i,g_j)*D = (g_i,g_j) \cdot D$ for all $(g_i,g_j) \in G_1 \times G_2$. Since the underlying sets of $G_1 \rtimes_\phi G_2$ and $G_1 \times G_2$ are equal, form the incidence matrix $A$ for the development of $D$ in $G_1 \rtimes_\phi G_2$ and the incidence matrix $A'$ for the development $D$ in $G_1 \times G_2$ with the same ordering of the underlying set. Consider the incidence of some arbitrary point $y$ on some arbitrary block $rD$ in both designs. Then since $r,y \in G_1 \times G_2$, we have that $r = (g_i,g_j)$ and $y = (g_k, g_l)$ for some $(g_i,g_j), (g_k,g_l) \in G_1 \times G_2$. Thus since $r*D = r\cdot D$ by supposition, the incidence of $y$ in $r*D$ is the same as the incidence of $y$ in $r \cdot D$. Thus, since $r,y$ were arbitrary, $A$ must equal $A'$, and so the designs are isomorphic. Also, $D$ is a symplectic difference set in $G_1 \times G_2$ by the definition of the symplectic design, so $D$ is also a symplectic difference set in $D_1 \rtimes_\phi D_2$.
\end{proof}
This condition by itself is somewhat nebulous: for what (nontrivial) choices of $\phi$ is this the case? Firstly, note that we need only work with $(1,g_j)*D$, since $G_1$ applies none of the ``twisting" from the semi-direct product. Further, $\phi_{g_j}$ must induce a permutation on $D_1$ for all $g_j$. This leads us to the following theorem:

\begin{theorem}\label{theorem:semi-direct condition}
Let $G_1$ and $G_2$ be groups of even power of 2 order with $G_2$ having generators $\{x_1,x_2,\cdots,x_n\}$, and let $D_1$ and $D_2$ be symplectic difference sets in $G_1$ and $G_2$, respectively. Let $\phi:G_2 \to Aut(G_1)$ be a homomorphism. Then $D = \big(D_1 \times (G_2-D_2)\big) \cup \big((G_1-D_1) \times D_2\big)$ is a symplectic difference set in $G_1 \rtimes_\phi G_2$ if $\phi_{x_i}(D_1)=D_1$ for all $x_i$.
\end{theorem}
\begin{proof}
Suppose that $\phi_{x_i}(D_1)=D_1$ for all $x_i$. Let $g = x_1^{e_1}x_2^{e_2}\cdots x_n^{e_n} \in G_2$ be arbitrary. Then since $\phi$ is a homomorphism, $\phi_g(D_1) = \phi_{x_1^{e_1}}\phi_{x_2^{e_2}}\cdots\phi_{x_n^{e_n}}(D_1) = D_1$ by supposition.
\end{proof}
Thus, we can determine whether the condition in Theorem \ref{theorem:semi-direct condition} is met using only the generators of $G_2$. This significantly decreases the search space for such homomorphisms, as if there are $m_1$ mappings of $x_1$, $m_2$ mappings of $x_2$, and so on up to $m_n$ mappings of $x_n$ that keep $D$ fixed, there are then $\prod_{i=1}^n m_i$ homomorphisms that satisfy the property in Theorem \ref{theorem:semi-direct condition}. We present the following as a non-trivial example of this property:

\begin{example}
Let $G = (C_8 \times C_2) \rtimes_\phi C_4$, ($x^8=y^2=z^4=1$), where $\phi$ is defined by $\phi_z(x) = x^5$, $\phi_z(y)=y$. Recall that $C_8 \times C_2$ has a symplectic SDP difference set $D = \{1,x,x^2,x^5,y,x^6y\}$, and $C_4$ has the trivial symplectic SDP difference set $\{1\}$. Note that $\phi_z(D) = \{1, x^5, (x^5)^2, (x^5)^5, y, (x^5)^6y\} = \{1, x^5, x^2, x, y, x^6y\}$, and thus since the automorphism induced by the generator of $C_4$ fixes $D$, we must have that the incidence matrix of $(D \times \{z,z^2,z^3\}) \cup ((C_8 \times C_2 - D) \times \{1\})$ under this semi-direct product is equal to the incidence matrix under the direct product, and so the design is symplectic.
\end{example}

\section{Isomorphic Designs under Product Construction}\label{section:isomorphisms}

The closure of the symmetric difference property under direct product construction also opens a new line of questioning: how does product construction interact with isomorphic designs? In particular, given two SDP difference sets $D,D'$ produced from the product construction of isomorphic designs, how do different semi-direct products effect the equivalence of $D$ and $D'$? If the developments of two difference sets $D_1$ and $D_2$ are isomorphic, by an abuse of notation we denote this relation by $D_1 \simeq D_2$. For the direct product, we present the following theorem:

\begin{theorem}\label{theorem:direct-iso}
Let $G = G_1 \times G_2$ and $G' = G_1' \times G_2'$ be two distinct groups of order $2^{2n}$ with respective SDP difference sets $D = \big(D_1 \times (G_2-D_2)\big) \cup \big((G_1-D_1) \times D_2\big)$ and $D' = \big(D_1' \times (G_2'-D_2')\big) \cup \big((G_1'-D_1') \times D_2'\big)$ for SDP difference sets $D_i \in G_i$ and $D_i' \in G_i'$. Then if $D_1 \simeq D_1'$ and $D_2 \simeq D_2'$, then $D \simeq D'$.
\end{theorem}
\begin{proof}
Suppose $D_1 \simeq D_1'$ and $D_2 \simeq D_2'$. Then $|G_1| = |G_1'|=v_1$ and $|G_2|=|G_2'|=v_2$. Let $A_1, A_1', A_2, A_2'$ be the incidence matrices for the developments of $D_1,D_1',D_2,D_2'$, respectively, and maintain the ordering of $G_1, G_1', G_2, G_2'$ as in these incidence matrices for the remainder of the proof. Then using the same ordering of $G$ as in Theorem \ref{theorem:direct-product}, we have $A$ and $A'$ being block matrices made of copies of $A_1,A_1^c$ and $A_1',A_1^{\prime c}$, respectively. Since the designs of $D_1$ and $D_1'$ are isomorphic, there exist permutation matrices $P_1,Q_1$ such that $A_1' = P_1 A_1 Q_1$, and likewise there exist permutation matrices $P_2, Q_2$ for $A_2$ into $A_2'$. Then consider the diagonal block matrices $P_1',Q_1'$ with $P_1$ on the diagonal of one and $Q_1$ on the diagonal of the other. Then $P_1'AQ_1'$ permutes each $A_1$ or $A_1^c$ in $A$ into $A_1'$ or $A_1^{\prime c}$, respectively. Now, recalling that $A$ is $A_2$ as a block matrix, let $P_2'$ and $Q_2'$ be the block matrix equivalents (zero $v_1 \times v_1$ matrix in place of 0, and $I_{v_1}$ in place of 1) of $P_2$ and $Q_2$, respectively. Then $(P_2'P_1')A(Q_1'Q_2') = A'$, and we thus have that $D \simeq D'$.
\end{proof}

From this theorem, we have the following corollaries:

\begin{cor}
The developments of all direct product-constructed SDP difference sets of order 64 are isomorphic to the symplectic design on 64 points.
\end{cor}
\begin{proof}
We first note the SDP on $C_2^6$ by iterative product construction is isomorphic to the symplectic design on 64 points. Let $T$ be the trivial SDP difference set on $C_2^2$, and let $S$ be the product constructed SDP difference set on $C_2^4$. Note as well that there is only one SDP design each on 16 and 4 points. Let $G_1$ and $G_2$ be arbitrary groups of order 16 and 4 that contain SDP difference sets $D_1$ and $D_2$, respectively. Then since there is only one SDP design on 16 points and 4 points, $D_1 \simeq S$ and $D_2 \simeq T$. Thus, by Theorem \ref{theorem:direct-iso}, the product construction of $D_1$ and $D_2$ is isomorphic to the symplectic design.
\end{proof}

\begin{cor}
There are four nonisomorphic SDP designs on 256 points coming from product construction.
\end{cor}
\begin{proof}
There are four nonisomorphic SDP designs on 64 points and one SDP design on 4 points, and so there is exactly one nonisomorphic design on 256 for each nonisomorphic design on 64 points by Theorem \ref{theorem:direct-iso}.
\end{proof}

\bibliographystyle{amsplain}

\end{document}